\documentclass[12pt]{iopart}
\usepackage{amsfonts}
\usepackage{graphicx}
\usepackage{amssymb}
\usepackage{amscd}
\usepackage{amsthm}
\usepackage{eqnarray}
\usepackage[utf8]{inputenc}
\usepackage{algorithm}
\usepackage[noend]{algpseudocode}
\newtheorem{theorem}{Theorem}[section]
\newtheorem{corollary}[theorem]{Corollary}

\newtheorem{definition}[theorem]{Definition}

\newtheorem{lemma}[theorem]{Lemma}

\newtheorem{problem}[theorem]{Problem}
\newtheorem{proposition}[theorem]{Proposition}

\usepackage[a4paper=truepdftex,bookmarksnumbered, colorlinks, plainpages]{hyperref}
\hypersetup{
	colorlinks=true,
	linkcolor=blue,
	urlcolor=red,
	citecolor=red,
	linktoc=all
}
\def\defeq{\mathrel{\mathop:}=}
\begin{document}
	\title{On a numerical construction of doubly stochastic matrices with prescribed eigenvalues}
	\author{Kassem Rammal$^1$, Bassam Mourad$^2$, Hassane Abbas$^2$, Hassan Issa$^2$}
	\address{$^1$KALMA, Department of Mathematics, Faculty of Science, Lebanese University, Beirut, Lebanon}
\ead{kassem.rammal@hotmail.com }
	\address{$^2$Department of Mathematics, Faculty of Science, Lebanese University, Beirut, Lebanon}
\ead{bmourad@ul.edu.lb}
	\date{}
\begin{abstract}
We study the inverse eigenvalue problem for finding doubly stochastic matrices with specified eigenvalues. By making use of a combination of Dykstra's algorithm and an alternating projection process onto a non-convex set, we derive hybrid algorithms for finding doubly stochastic matrices and symmetric doubly stochastic matrices with prescribed eigenvalues. Furthermore, we prove that the proposed algorithms converge, and linear convergence is also proved. Numerical examples are presented to demonstrate the efficiency of our method.
\end{abstract}

\maketitle
	\section{Introduction}
	An $n\times n$ matrix is said to be \emph{doubly stochastic} if it has nonnegative entries and the sum of the entries in every row and every column is equal to one.
	Doubly stochastic matrices have many applications such as in operator algebras and quantum physics  an important role for them is presented  in the book \cite{Alberti}. Many interesting problems  in engineering and robotic problems such as motion planning, localization and navigation require the construction of doubly stochastic matrices. For example,  in robotic networks, the discrete time consensus algorithm requires the adjacency matrix to be doubly stochastic (see~\cite{Atrianfar,Bullo,Rivaz}).
	A recent application of doubly stochastic matrices appears in \cite{Genetic} where the authors showed a close relation between the genetic code and doubly stochastic matrix by using Hamming distance via the Gray code correspondence. In addition, doubly stochastic matrices play a significant role in probability theory, communication theory of satellite-switched, time-division, multiple-access systems, quantum mechanics, graph theory, graph-based clustering, the assignment problem, and into the Schur-Horn convexity theorem, etc. We refer the interested reader to \cite{CG_Method_2020} for more details and the references therein.
	
	Recall that the inverse eigenvalue problem for special kind of matrices is concerned with constructing a matrix that maintains the required structure from its set of eigenvalues.
	Constructing matrices for the inverse eigenvalue problems arise in a remarkable variety of applications, including system and control theory, geophysics, molecular spectroscopy, quantum mechanics, particle physics, structure analysis, economics and operation research to name a few  (\cite{Golub_05}). See also \cite{Smo} for a recent application.
	
	The nonnegative inverse eigenvalue problem (NIEP) asks which sets of $n$ complex numbers can occur as the spectrum of an $n\times n$ nonnegative matrix $A$.
	As mentioned in \cite{Egleston_2004}, many sub-problems have emerged from the NIEP because of its complexity. One of these is the symmetric nonnegative inverse eigenvalue problem (SNIEP) which deals with finding which sets of $n$ real numbers serve as the spectrum of an $n\times n$ symmetric nonnegative matrix $A$.
	Another related problem is the stochastic inverse eigenvalue problem which is concerned with  determining necessary and sufficient conditions for a list of $n$ complex numbers to be the spectrum of a stochastic $n \times n$ matrix (an $n \times n$ matrix is called stochastic if it is entry-wise nonnegative and if the sum of the entries in each row equals one).
	However, Johnson \cite{John_1981} shows that if a nonnegative matrix $A$ has a positive Perron root $\rho$ and corresponding positive eigenvector $v,$ then the matrix $\rho^{-1}\textnormal{diag}(v)^{-1}A \ \textnormal{diag}(v)$ is a row stochastic matrix. Thus, the stochastic inverse eigenvalue problem and the (normalized) NIEP are equivalent.
	
	The doubly stochastic inverse eigenvalue problem (hereafter, DIEP) concerns the reconstruction of a doubly stochastic matrix from certain prescribed eigenvalues.
	Moreover, the \emph{symmetric doubly stochastic inverse eigenvalue problem} (hereafter, SDIEP) can be stated as follows:
	given $ \lambda=\left(1,\lambda_2,\ldots,\lambda_n\right)$, find an $n \times n$ symmetric doubly stochastic  matrix with spectrum $\lambda$ (if such a matrix exists). It is worth noting here that the DIEP and the NIEP are not equivalent  \cite{John_1981}.
	For more details concerning these problems see \cite{M1,M2,M3,M4,M5,M6}.
	
	In the language of \cite{Golub_05}, we have two  challenging issues  for  DIEP.  The first one deals with  finding all the necessary and sufficient conditions for its solvability. The second challenge lies in knowing that such a matrix exists, how to construct it numerically.
	
	A list $ \lambda=\left(1,\lambda_2,\ldots,\lambda_n\right)$ of (real) numbers is said to be (s.d.s) d.s. realizable if there exists an $n\times n$ (symmetric) doubly stochastic matrix $A$ with $\lambda$ as its spectrum. Some of the most broad-based necessary conditions for a list $\lambda=\left(1,\lambda_2,\ldots,\lambda_n\right)$ to be d.s. realizable, are given by:
	\begin{enumerate}
		\item From Perron-Frobenius theorem (see~\cite{Minc_88}), $1$ is the maximal positive eigenvalue (Perron eigenvalue) which is also  greater than or equal to the modulus of each of the other eigenvalues. To this eigenvalue corresponds the eigenvector $\left(1,\ldots,1\right)^T$ which is also referred to as the Perron-Frobenius eigenvector.
		
		\item $\lambda=\overline{\lambda}.$
		\item $s_k=1+\sum_{i=2}^{n}\lambda_i^k \ge 0, \quad k=1,2,\ldots\cdot$
		\item  $s_k^m\le n^{m-1}s_{km}$ for all $k, m=1,2,\ldots$ (such inequalities are referred to as the JLL conditions \cite{Minc_88}).
	\end{enumerate}
	We now set the scene for more notation. For $\lambda=(1,\lambda_2,...,\lambda_n)$, let
	\begin{equation}
		\Lambda=\textnormal{diag}(\lambda)=\textnormal{diag}\left(1,\lambda_2,\ldots,\lambda_n \right),\label{eq:decreas_lambda}
	\end{equation}
	and let $\textbf{M}$ be the set of all real $n\times n$ matrices with eigenvalues $\lambda$ defined by
	\begin{equation*}
		\textbf{M}=\{A\in M_n\left( \mathbb{R}\right) \colon A=P\Lambda P^{-1} \textnormal{ for some nonsingular } P\}.
	\end{equation*}
	Let $\textbf{N}$ be the cone of nonnegative $n\times n$ matrices defined by
	\begin{equation*}
		\textbf{N}=\{A\in M_n\left(\mathbb{R}\right)\colon a_{ij}\ge 0\quad \textnormal{for all} \ i,j\}.
	\end{equation*}
	Let $e=\left(1,\ldots,1\right)^T$ be in $\mathbb{R}^n$ and  let $\widehat{\Omega}_{n}$ denote the set of generalized doubly stochastic $n\times n$ real matrices,
	\begin{equation*}\widehat{\Omega}_{n}=\{A\in M_n\left( \mathbb{R}\right)\colon Ae=e \textnormal{ and } A^Te=e\}.
	\end{equation*}

	In 2006, Orsi \cite{Orsi} developed a novel algorithm to solve the (symmetric) nonnegative inverse eigenvalue problem NIEP (SNIEP) using an alternating projection method on the two sets $\textbf{M}$ and $\textbf{N}$.  Earlier, some important algorithmic approaches were also developed in \cite{Chu_1998} for NIEP and SNIEP based on  least squares methods that can be formulated as constrained optimization problems. However, most of  these methods only apply to their respective feasible region but not to the general case. So this is one of the reasons why the SNIEP and NIEP have not been solved completely.

	In this paper, we are interested in finding numerical solutions for both DIEP and SDIEP.  Many exploratory works on the DIEP have been established, such as monotone and non-monotone Riemannian inexact Newton-CG methods \cite{CG_Method_2020} and Riemannian optimization approach to solve a subset of convex optimization problems wherein the optimization variable is a doubly stochastic matrix \cite{Douik_Graph,Douik_geo}. Also, in \cite{Chen_21}, the author considered another kind of DIEP, that is based on partial information on eigenvalues and eigenvectors. However, our aim  here  is to  further develop extensions of the methods used in \cite{Orsi} and \cite{dykstra_algo,dykstra} which is sort of  a ``combination'' of two algorithms virtually similar to those  mentioned in \cite{Orsi, dykstra_algo}.
	The rest of this paper is organized as follows. In Section 2, we study the alternating projections on particular subsets  of the three sets $\textbf{M, N}$ and $\widehat{\Omega}_{n}.$ This will require a careful study of the properties of the projections onto closed subspaces, closed linear varieties, closed convex sets or simply closed sets. This in turn will enable us to gain insight on  developing  an approximation algorithm that runs among the three above mentioned sets where  one of them is non-convex.

	In the third and  fourth sections, we propose a hybrid alternating projection algorithm to solve the SDIEP and DIEP. In Section 5, the convergence of the proposed algorithm is established under suitable conditions. Finally, we give some numerical experiments to evaluate the efficiency of the proposed method in Section 6.
	\section{Preliminary}
	We shall start with recalling some basic facts concerning the projection of a vector onto a closed subspace of a Hilbert space.
	\begin{theorem}[Orthogonal Projection Theorem]\label{th:Ortho_Proj}
		Let $H$ be a Hilbert space and $M$ a closed subspace of $H.$ For each vector $x \in H,$ there exists a unique vector $m_0 \in M$ such that $\|x-m_0\|\le\|x-m\|$ for all $m \in M.$ Moreover, a necessary and sufficient condition for $m_0 \in M$ to be the unique minimizing vector is that $x-m_0$ is orthogonal to $M.$
	\end{theorem}
	For convex sets, the following result is a direct extension of the preceding theorem.
	\begin{theorem}[Kolmogorov's criterion]\label{th:Kolmog_Proj}
		Let $x$ be a vector in a Hilbert space $H$ with inner product  $\langle \ . \  , \ . \ \rangle$ and let $C$ be a closed convex subset of $H.$ Then there exists a unique vector $c_0 \in C$ such that $\|x-c_0\|\le\|x-c\|$ for all $c \in C.$ Moreover, a necessary and sufficient condition that $c_0$ be the unique minimizing vector is that $\langle x-c_0,c-c_0\rangle\le 0$ for all $c \in C.$
	\end{theorem}
	Let $x$ be a vector in a Hilbert space $H$ and let $C$ be a closed subset of $H.$ Then, as usual a vector $c_0 \in C$ satisfies $\|x-c_0\|\le\|x-c\|$ for all $c \in C$ will be called a ``projection'' of $x$ onto $C.$ We will use $y = P_C(x)$ to denote that $y$ is a projection of $x$ onto $C$ though the uniqueness is not guaranteed.
	
	We have two ways to regard the DIEP (SDIEP). From an algebraic perspective, it is a structure problem, where one finds algebraically (symmetric) doubly stochastic matrix with prescribed eigenvalues. From a geometric point of view, the solutions of the DIEP or the SDIEP are the intersection of the three sets $\textbf{M,  N}$ and $\widehat{\Omega}_n.$ This immediately leads to a solution of the set intersection problem or the feasibility problem as follows,
	\begin{equation}
		\textnormal{Find}\quad x\in C=\textbf{M}\bigcap\textbf{N}\bigcap\widehat{\Omega}_n.
	\end{equation}	
	Moreover, we assume that for any $x \in H,$ the calculation of $P_C(x)$ is not trivial, whereas, for $P_{\textbf{M}}(x),P_{\textbf{N}}(x)$ and $P_{\widehat{\Omega}_n}(x)$ are easy to obtain.
	
	We note that $x\in \textbf{M}\bigcap\textbf{N}\bigcap\widehat{\Omega}_n$ is equivalent to $x\in \textbf{M}\bigcap\left( \textbf{N}\bigcap\widehat{\Omega}_n\right).$
	So that our method is based on alternative projections  onto the sets $\textbf{M}$ and $\left(\textbf{N}\bigcap\widehat{\Omega}_n\right)$ and then  finding a `best' projection.
	In order to find the projection onto $\left(\textbf{N}\bigcap\widehat{\Omega}_n\right),$ we shall present a variation of Von Neumann's scheme which was introduced by Dykstra as follows.
	
	Dykstra's algorithm was originally developed to project a given point in a finite-dimensional inner product space onto a closed convex cone. In \cite{dykstra,dykstra_algo}, Boyle and Dykstra describe how to find  the closest point $c_0$  to the intersection of several convex sets $C_i, i=1,\ldots,r,$ by generating two sequences, the iterates $\left\{x^n_i\right\}$ and the increments $\left\{I^n_i\right\},$ with $n \in \mathbb{N}$ and $i = 1, \ldots, r.$
	These sequences are defined by the recursive formulas
	\begin{eqnarray}\label{dyskstra_seq}
		x_0^n&=x_r^{n-1}, \ n\in \mathbb{N^*},\nonumber\\
		x_i^n&=P_{C_i}\left( x_{i-1}^n-I_i^{n-1}\right) \quad\textnormal{ and }\quad I^n_i=x_i^n-\left( x^n_{i-1}-I^{n-1}_i\right),
\end{eqnarray}
	for $i=1,\ldots,r$  with initial values $x_r^0=c_0, \ I^0_i=0$ and $ i=1,\ldots,r.$
	
	We note, that if some $C_j$ of the sets $C_i,$ $i = 1,\ldots, r,$ is a closed linear variety in $H,$ then the sequence $I_j^ n,$ of auxiliary quantities (``increments'') needs not to be computed, as in this case
	\begin{equation}\label{Proj_on_affine}
		x_j^n=P_{C_j}\left( x_{j-1}^n-I_{j}^{n-1}\right)=P_{C_j}\left( x_{j-1}^n\right),
	\end{equation}
	for all $n\in \mathbb{N}.$ To see why, set $C_j = b_j+L_j,$ where $b_j\in H$ and $L_j$ is a closed linear subspace of $H.$ Clearly for any $y \in H$ we have $P_{C_j}\left(y\right)= P_{L_j}\left(y - b_j\right) + b_j$ and hence
	\begin{equation*}
		P_{C_j}\left( y\right) -y=P_{L_j}\left(y - b_j\right) -\left( y- b_j\right) \in L_j^\perp,
	\end{equation*}
	where $L_j^\perp$ denotes the orthogonal complement of $L_j.$ Thus $$I_j^n=	P_{C_j}\left( x_{i-1}^n-I_i^{n-1}\right) -\left( x_{i-1}^n-I_i^{n-1}\right)\in L_j^\perp$$ for all $n \in \mathbb{N^*},$ and
	\begin{eqnarray*}
		x_j^n&=P_{L_j}\left( x_{j-1}^n-I_{j}^{n-1}-b_j\right)+b_j\\
		&=P_{L_j}\left( \left( x_{j-1}^n-b_j\right) -I_{j}^{n-1}\right)+b_j\\
		&=P_{L_j}\left( x_{j-1}^n-b_j\right)+b_j \qquad \textnormal{(as }
		P_{L_j} \textnormal{ is linear)}\\
		&=P_{C_j}\left( x_{j-1}^n\right).
	\end{eqnarray*}
	
	For our purposes, the following theorem guarantees that Dykstra's algorithm converges to the closest point.
	\begin{theorem}[Boyle and Dykstra~\cite{dykstra}]\label{th::dykstra}
Let $C_1, \ldots ,C_r$ be closed and convex subsets of a Hilbert space $H$ such that $C = \bigcap_{i=1}^{r}C_i\ne \emptyset.$ For any $i = 1, \ldots , r$ and any $x_0 \in H,$ the sequence $\left\{x^n_i\right\}$ generated by \ref{dyskstra_seq} converges strongly to $x^* = P_C\left(x_0\right)$ (i.e.  $\| x^n_i- x^*\| \longrightarrow 0$ as $n\rightarrow\infty$).
	\end{theorem}
	Next, we consider the following problem of interst:
	\begin{problem}\label{prob:1}
		Find $x\in C\defeq\bigcap_{i=1}^{r}C_i\ne\emptyset,$ where $C_i$ is closed set in a Hilbert space.	
	\end{problem}
	If in addition, the involved sets in Problem~\ref{prob:1} are subspaces of a Hilbert space, then Problem~\ref{prob:1} can be solved algorithmically using the method of alternating projections (MAP) which dates back to John Von Neumann \cite{Von_Neumann_33} or its Halperin's extension \cite{Halperin_62}.
	Later, Cheney and Goldstein \cite{Cheney_59} and Bregman \cite{MAP} extended the analysis of Von Neumann's alternating projection scheme to the case of two closed and convex sets.
	\begin{theorem}[see \cite{Cheney_59,MAP}]
		Let $C_1, \ldots , C_r$ be closed and convex subsets of a finite dimensional Hilbert space $H$ with nonempty intersection $C=\bigcap_{i=1}^{r} C_i.$ Then, for some $x_0\in H,$ we have that
		\begin{equation*}
			\lim\limits_{n\to \infty}\left(P_{C_r}P_{C_{r-1}}\cdots P_{C_1}\right)^nx_0\in C.
		\end{equation*}
	\end{theorem}
	Note that the limit point of the sequence generated by any of the alternating projections methods discussed above does not need to be the closest in the intersection to the given initial point. Therefore, all the previously mentioned MAPs are not useful for solving certain optimization problems for which this optimal property (convexity) is crucial. However, Dykstra's method is not limited to using convex sets for which convergence to the closest point is guaranteed.
	
	In practice, when we work with a feasible region given by the intersection of closed subspaces (or affine sets), then Dykstra's method and Von Neumann's MAP coincide. Moreover, as pointed out
	earlier, there is no need to consider the increments associated with the subspaces that appear in the definition of the feasible region.
	
	Our presentation will be more general for finding the closest point onto the intersection of three sets at least one of them is non-convex. This leads to study the projection of an element onto a closed set and their general properties that will be used throughout the paper.
	\begin{proposition}[see~\cite{Bauschke_2013}] \label{Prop:Reduc_Dist}
		Let $C_1$ and $C_2$ be closed (nonempty) sets in a finite dimensional Hilbert space $H.$
		If
		\begin{equation*}
			x_{i+1} = P_{C_1} (y_i), \ y_{i+1} = P_{C_2} (x_{i+1}),
		\end{equation*}
		for $i = 0, 1, \ldots $ and if $\left\{x_n\right\}_n\cap C_2\ne \emptyset$ or $\left\{y_n\right\}_n\cap C_1\ne \emptyset,$ then there exists $\overline{x}\in C_1 \cap C_2$ such that for all $n$ sufficiently large, $x_n= y_n=\overline{x}.$
	\end{proposition}
	In this presentation, we apply the technique of alternating projection between three sets, $C_1, C_2$ and $C_3,$ when $C_1$ is non-convex set. In the event that projections onto these sets are computable, then  an answer strategy  is to alternately project onto $C_1, C_2\cap C_3$ on the one hand and then on $C_2, \ C_3$ to find a point in their intersection, on the other hand.
	An illustration of our proposed alternating projections of point $A$ onto the sets $C_1, \ C_2$ and $C_3$ is given in Figure~\ref{fig:c1-c2-inter-c3-plans}.
	\begin{figure}[h]
		\centering
		\includegraphics[width=0.7\linewidth]{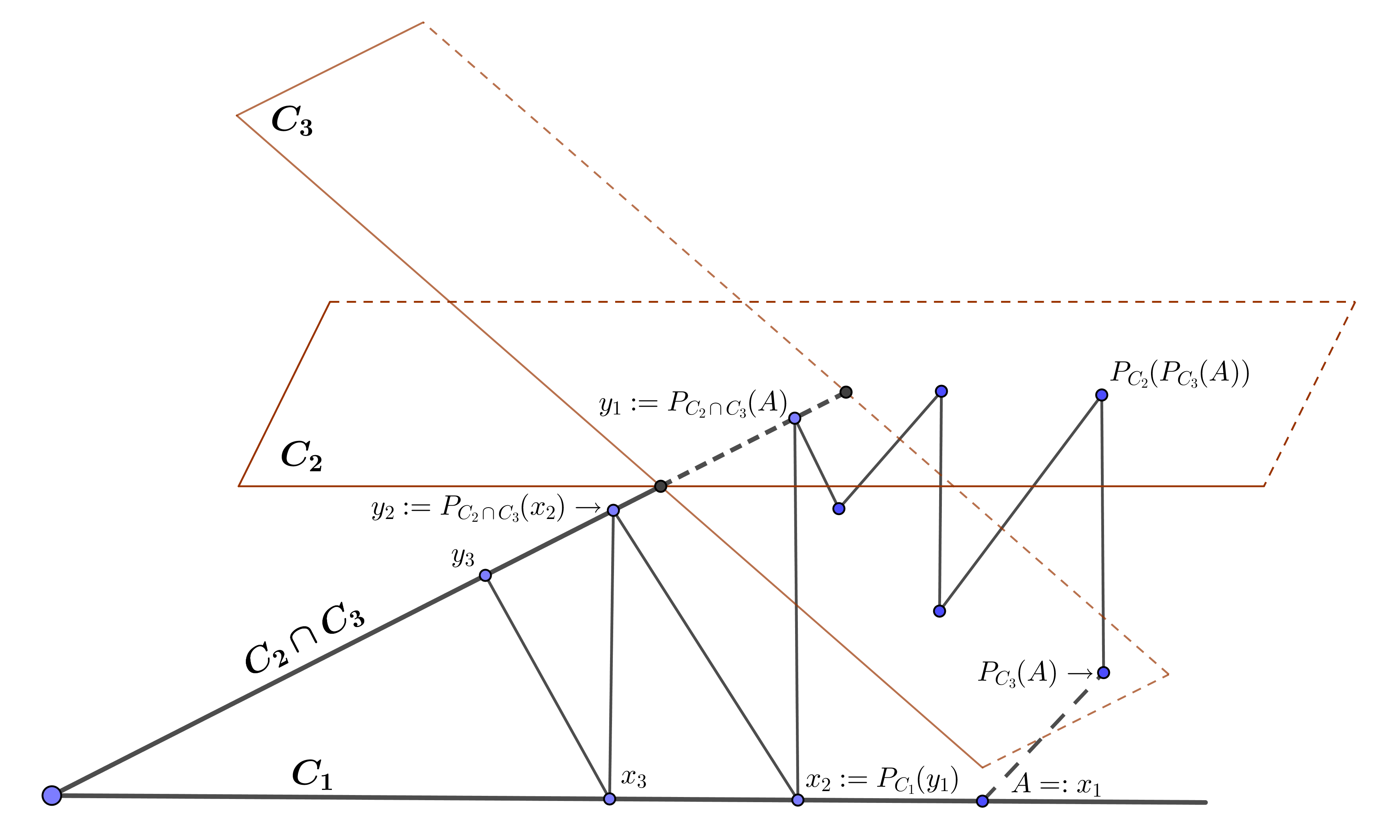}
		\caption[Alternating projections onto $C_1$ and $C_2\cap C_3.$]{Alternating projections onto $C_1$ and $C_2\cap C_3.$ More precisely, a cycle of projections scheme is performed alternatively onto every set separately and can be described as follows. Starting from a given matrix $x_1\defeq A\in C_1,$ then projecting onto $C_3$ to get $P_{C_3}\left(A\right)$. Next, starting with $P_{C_3}\left(A\right)$ we apply an alternating projection onto $C_2$ and $C_3$ until hitting a point $y_1\defeq P_{C_2\cap C_3}\left(x_1\right)$ in the intersection $C_2\cap C_3.$ Next, we project $y_1$ onto $C_1$ to obtain $x_2$ and we repeat the same process to get $y_2$ in $C_2\cap C_3$, and so on.}
		\label{fig:c1-c2-inter-c3-plans}
	\end{figure}
	If we can predict in advance that certain non-increasing distance $\|x_i - y_i\|$ are not guaranteed to reduce to zero, then we can ignore them in our algorithm.
	
	Let us turn to the Newton's identities (or the Girard-Newton formulae), which give relations between two types of symmetric polynomials, namely between Newton sums (or $k$-th moments) and elementary symmetric polynomials.
	\begin{definition}
		A polynomial $p\left(x_1,\ldots,x_n\right)$ is said to be symmetric if  for every permutation $\zeta$ of a set of $n$ objects, it holds that
		\begin{equation*}
			p\left(x_{\zeta\left(1\right)},\ldots,x_{\zeta\left(n\right)}\right)=p\left(x_1,\ldots,x_n\right).
		\end{equation*}
	\end{definition}
	The symmetric polynomials
	\begin{eqnarray*}
		\sigma_1\left(x_1,\ldots,x_n\right)&=\sum_{1\le i\le n}x_i,\\
		\sigma_2\left(x_1,\ldots,x_n\right)&=\sum_{1\le i<j\le n}x_ix_j.\\
		&\,\,\,\vdots\\
		\sigma_n\left(x_1,\ldots,x_n\right)&=x_1\cdots x_n
	\end{eqnarray*}
	are called the elementary symmetric polynomials in $x_1,\ldots,x_n.$
	
	The following identities gives the desired relationships.
	\begin{theorem}[{Newton-Girard formulas,~\cite{Seroul_00}}]
		Let $\lambda=\left(\lambda_1,\ldots,\lambda_n\right)$ be an $n$-tuple of complex numbers.
		For $1\le k\le n,$ we have
		\begin{eqnarray*}
			0&=s_1\left(\lambda\right)-\sigma_1\left(\lambda\right).\\
			0&=s_2\left(\lambda\right)-s_1\left(\lambda\right)\sigma_1\left(\lambda\right)+2\sigma_2\left(\lambda\right).\\
			&\,\,\,\vdots\\
			 0&=s_n\left(\lambda\right)-s_{n-1}\left(\lambda\right)\sigma_1\left(\lambda\right)+\cdots+(-1)^{n-1}s_1\left(\lambda\right)\sigma_{n-1}\left(\lambda\right)+(-1)^nn\sigma_n\left(\lambda\right).
		\end{eqnarray*}
		For  $k>n,$ set $k=n+i$ with $i>0$  and then it holds that
		\begin{equation*}
			s_{n+i}\left(\lambda\right)-s_{n+i-1}\left(\lambda\right)\sigma_1\left(\lambda\right)+\cdots+(-1)^ns_i\left(\lambda\right)\sigma_n\left(\lambda\right)=0.
		\end{equation*}
	\end{theorem}
	Obviously, the characteristic polynomial of a nonnegative matrix has real coefficients and for later use, we conclude this section with the Cayley-Hamilton theorem that  can be stated as follows.
\begin{theorem}[Hamilton-Cayley,~\cite{Seroul_00}]
	Let $A$ be an $n\times n$ matrix with coefficients in a
	commutative ring with unit. If
	$\det\left(x I_n-A\right)=\sum_{k=0}^{n}\alpha_kx^{k}$ is the characteristic polynomial of matrix $A,$ then the following matrix equation holds:
	\begin{equation}\label{eq:Cayley_Matrix}
		\sum_{k=0}^{n}\alpha_kA^{k}=0.
	\end{equation}	
\end{theorem}
\section{The symmetric doubly stochastic problem}
In this section, in order to solve SDIEP we shall apply the symmetric version of the above  recursive algorithm which   consists of alternately projecting onto three particular sets.
But first we need to introduce the relevant terminology for the symmetric case. Let $\mathcal{S}^n$ denote the set of all real symmetric $n\times n$ matrices.
As usual, $\mathcal{S}^n$ is considered as a Hilbert space with inner product
\begin{equation*}
	\langle A,B\rangle=\textnormal{ tr}\left( AB\right) =\sum_{i,j}a_{ij}b_{ij},
\end{equation*}
and the associated norm is the Frobenius norm $\| A\|_F=\langle A, A\rangle^{\frac{1}{2}}.$

For a list  $\lambda=\left(1,\lambda_2,\ldots,\lambda_n\right)$  of $n$ real numbers, we shall assume that $1\ge\lambda_2\ge\cdots\ge\lambda_n$ (renumbering if necessary) and we denote
\begin{equation}
	\Lambda=\textnormal{diag}\left(1,\lambda_2,\ldots,\lambda_n \right).
\end{equation}
In addition, we let $\mathcal{M}$ be the set of all real symmetric $n\times n$ matrices with spectrum $\lambda$ i.e.
\begin{equation*}
	\mathcal{M}=\{A\in \mathcal{S}^n \colon A=V\Lambda V^T \textnormal{ for some orthogonal } V\}.
\end{equation*}
Let $\mathcal{N}$ be the set of symmetric nonnegative $n\times n$ matrices i.e.
\begin{equation*}
	\mathcal{N}=\{A\in \mathcal{S}^n\colon a_{ij}\ge 0\quad \forall \ i,j\}.
\end{equation*}
Let ${\Omega}_{n}$ denote the set of all $n\times n$ doubly stochastic  matrices,
\begin{equation*}
	{\Omega}_{n}=\widehat{\Omega}_n\cap \mathcal{N}=\{A\in M_n\left(\mathbb{R}\right)\colon Ae=e, A^Te=e \textnormal{ and } a_{ij}\geq 0\quad \forall \  i,j\}.
\end{equation*}
Let $\widehat{\Omega}_{n}^s$ and $\Omega_{n}^s$ denote the set of symmetric generalized doubly stochastic $n\times n$ real matrices and the set of symmetric doubly stochastic $n\times n$ matrices respectively.

It follows that solving the SDIEP with prescribed eigenvalues $\lambda$ is equivalent to finding an $n\times n$ matrix $X$ that belongs to the intersection of $\mathcal{M}, \mathcal{N}$ and $\widehat{\Omega}_{n}^s.$ Thus, the following problem is considered:
\begin{equation*}
	\textnormal{ Find } X\in \mathcal{M}\cap\left(\mathcal{N}\cap\widehat{\Omega}_{n}^s\right).
\end{equation*}
As mentioned earlier, our method depends on alternative projections by finding a `best' projection onto the two sets $\mathcal{M}$ and $\left(\mathcal{N}\cap\widehat{\Omega}_{n}^s\right)$ instead of the three sets $\mathcal{M, \ N}$ and $\widehat{\Omega}_n.$
Indeed, we first apply Boyle and Dykstra's alternating projections to find the nearest symmetric doubly stochastic matrix to a given real matrix. This means that the outcome matrix is the projection of a given real matrix onto $\left(\mathcal{N}\cap\widehat{\Omega}_{n}^s\right).$ However, as $\mathcal{M}$ is non-convex then the projection of a matrix $A$ onto the set $\mathcal{M}$ may not be unique.

Using Theorem~\ref{th:approx_to_M}, a best projection of a symmetric matrix $A$ onto the set $\mathcal{M}$ can be achieved.
\begin{theorem}[see \cite{Orsi}] \label{th:approx_to_M}
	Let $V \textnormal{diag}(\mu_1,\ldots,\mu_n)V^T$ be a spectral decomposition of $A$ with $\mu_1\ge\cdots \ge \mu_n.$ Then $V\textnormal{diag}(1,\lambda_2,\ldots,\lambda_n) V^T$ is a best approximate in $\mathcal{M}$ to $A$ in the Frobenius norm.
\end{theorem}
The key ingredient of the proof of the preceding theorem is the following observation due to Hoffman and Wielandt \cite{Golub,Horn} which also has been clarified by Brockett and Chu in their papers \cite{Brockett,Chu}, respectively.	
\begin{theorem}\label{Hoffman_th}
	Let $A, B\in M_{m, n}\left(\mathbb{C}\right)$ with $m\ge n$ be given and suppose that $\tau_1,\ldots,\tau_n$ and $\sigma_1,\ldots,\sigma_n$ are the non-increasingly ordered singular values of $A$ and $B,$ respectively. Then
	\begin{equation*}
		\| A-B\|_F^2\ge \sum_{i=1}^{n}\left( \tau_i-\sigma_i\right)^2.
	\end{equation*}
\end{theorem}
The projection onto the convex set $\mathcal{N}$ is well known and it is given in the following lemma where for completeness we include its proof here.
\begin{lemma}\label{th:proj_N}
	Let $A \in M_n\left(\mathbb{R}\right)$ and let $A_+ \in M_n\left(\mathbb{R}^+\right)$ be a matrix defined by
	\begin{equation*}
		\left( A_+\right)_{ij}=
\cases{a_{ij},  &\textnormal{if } $a_{ij}\ge 0$,\\
0, & \textnormal{if} $a_{ij}<0$,}
	\end{equation*}
	for all $1 \le i, j \le n.$
	Then $A_+$ is the best approximate in the set of  nonnegative $n\times n$ matrices to $A$ in the Frobenius norm. This is especially useful in $\mathcal{N}\subset \mathcal{S}^n.$
\end{lemma}
\begin{proof}
	Let $N=\left(n_{ij}\right)$ be a nonnegative matrix. Then
	\begin{eqnarray*}
		\| A-A_+ \|_F^2&=\sum_{a_{ij}<0}\left| a_{ij}\right|^2\\
		&\le \sum_{a_{ij}<0}\left| a_{ij}-n_{ij}\right|^2\\
		&\le \sum_{i,j}\left| a_{ij}-n_{ij}\right|^2\\
		&=\| A-N \|_F^2.
	\end{eqnarray*}
	This completes the proof.
\end{proof}

Let $J_n$ be the $n\times n$ matrix whose entries are all equal to $\frac{1}{n}$ and let $I_n$ be the $n\times n$ identity matrix in $M_n\left(\mathbb{R}\right).$

Due to \cite{Khoury}, the projection onto $\widehat{\Omega}_n$ is straightforward and is given by the following theorem.
\begin{theorem}
	Let $A \in M_n\left(\mathbb{R}\right)$ and let $B_A$ be an $n\times n$ real matrix defined by
	\begin{equation*}
		B_A = W_nAW_n+J_n,
	\end{equation*}
	where $W_n=I_n-J_n.$ Then $B_A$ is the best approximate in $\widehat{\Omega}_{n}$ to $A$ in the Frobenius norm.
\end{theorem}

\begin{corollary}\label{cor:omega_sym}
	Let $A\in \mathcal{S}^n.$ Then $B_A=W_nAW_n+J_n\in\widehat{\Omega}^s_n,$ where $W_n=I_n-J_n.$
\end{corollary}

A first observation is that the set $\mathcal{N}$ is a convex cone and the set $\widehat{\Omega}_n$ is an affine subspace.
A second observation is that for a matrix $A \in M_n\left(\mathbb{R}\right),$  the best approximate in ${\Omega}_{n}$ to $A$ in the Frobenius norm can be found by applying Theorem~\ref{th::dykstra} so that $C=\Omega_{n}=\widehat{\Omega}_{n}\cap\mathcal{N}$ and thus $r=2, x_0^1=A, x_1^1=B_A, x_2^1=\left( B_A\right) _+,$ and so on. The intended construction is illustrated in the following figure.

\begin{figure}[h]
	\centering
	\includegraphics[width=0.7\linewidth]{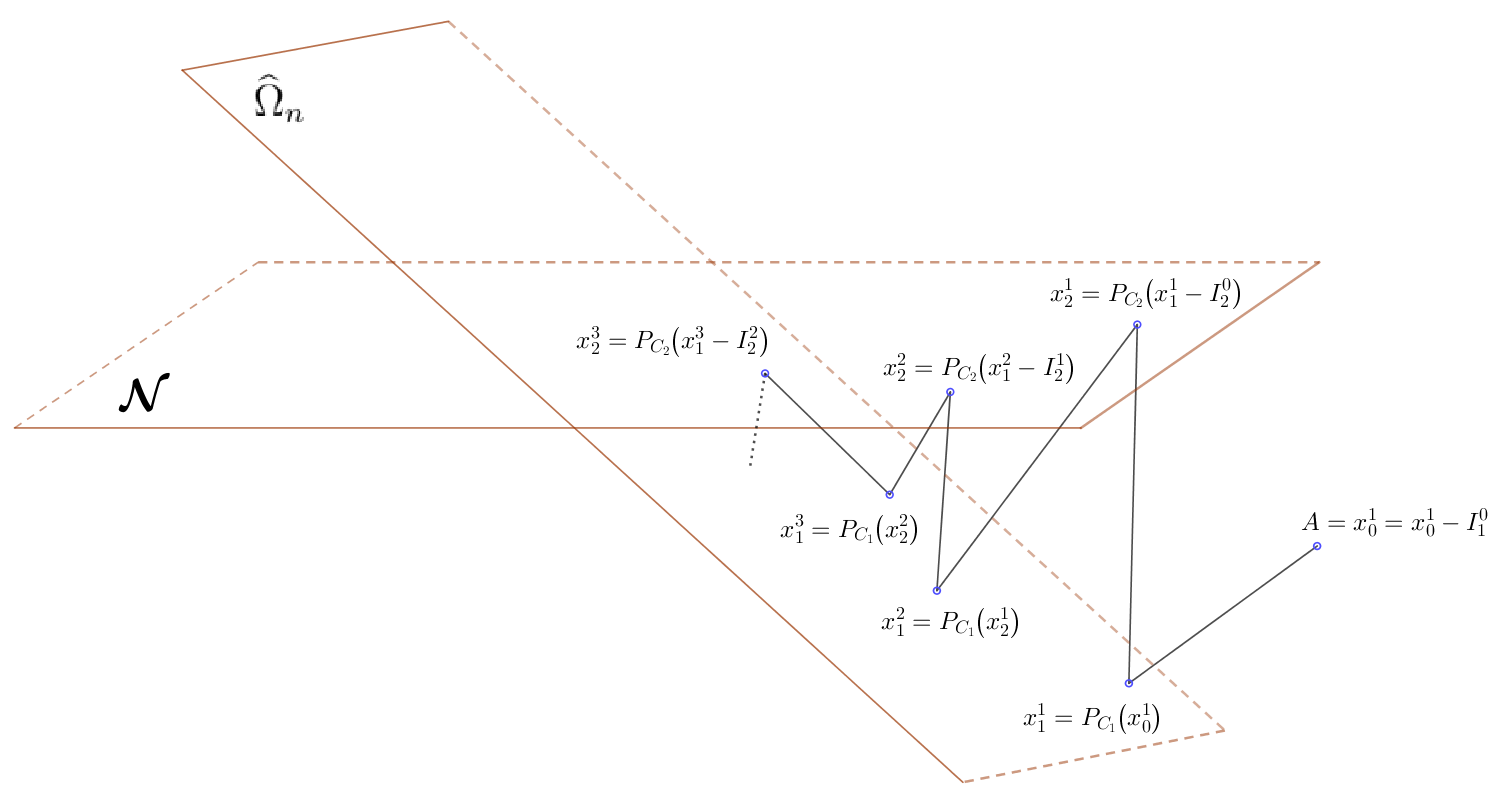}
	\caption[Dykstra's Algorithm onto $\widehat{\Omega}_{n}$ and $\mathcal{N}$]{Dykstra's Algorithm onto $\widehat{\Omega}_{n}$ and $\mathcal{N}$ starting from a chosen matrix $x_0^1\defeq A\in \mathcal{M} $ then projecting alternatively onto the two cones. In this way, a sequence of elements that converges to $P_{\mathcal{N}\cap \widehat{\Omega}_{n}}\left(A\right)$ is generated. More explicitly, the starting point $A$ is projected to $x_1^1$ in $\widehat{\Omega}_{n}$, $x_1^1$ is then projected to $x_2^1$ in $\mathcal{N}$. Next, $x_2^1$ is projected to $x_1^2$ in $\widehat{\Omega}_{n}$ and then $x_1^1$ is projected to $x_2^2$ in $\mathcal{N}$, and so on. }
	\label{fig:Dykstra's r=2}
\end{figure}

Boyle and Dykstra~\cite{dykstra} generalized the simple alternating projections with a clever way to an algorithm which converges to the nearest point in the intersection (of a finite number of closed convex sets) to the initial point.

 Here is a proposed special case of the Dykstra's algorithm with $r=2$ presented in pseudo-code:

\begin{algorithm}[h]
	\caption{Dykstra's algorithm for finding the nearest D.S matrix}
	\label{algo::Real_Dykstra}
	\begin{algorithmic}[1]
		\Require An $n\times n$ real matrix $A$
		\Procedure {Return}{$D.S$}
		\State $X\defeq 0, Y\defeq A, I_2^0\defeq0$
		\State $k\defeq 1$
		\While {$\| X-Y\| \ge \varepsilon$}
		\State $X\defeq WYW+J_n$\Comment{Due to equality~\ref{Proj_on_affine}, since $\widehat{\Omega}_{n}$ is an affine subspace}	
		\State $Y\defeq\left(X-I_2^{k-1} \right)_+$
		\State $I_2^{k}\defeq Y-\left( X-I_2^{k-1}\right) $
		\State $k\defeq k+1$
		\EndWhile
		\EndProcedure\Comment{Output: $P_{\Omega_{n}}\left( A\right)$ is a doubly stochastic matrix}
	\end{algorithmic}
\end{algorithm}

\vspace{4cm}

Based on the previous discussion and Algorithm~\ref{algo::Real_Dykstra}, we propose a hybrid alternating projection method to solve the SDIEP given in the following algorithm:

\begin{algorithm}[h!]
	\caption{Algorithm for finding S.D.S matrix with prescribed spectrum}\label{algo:S.D.S}
	\begin{algorithmic}[1]
		\Require A symmetric nonnegative matrix $A$ and $\mathrm{Spec}\left(A\right) =\lambda$
		\Comment{Due to \cite[p. 195]{Orsi} }
		\Procedure {Return}{$S.D.S$}
		\State $Y\defeq A$
		\While {$\| X-Y\|\ge \varepsilon$}
		\State $Y\defeq V\textnormal{diag}\left( \mu_1,\ldots,\mu_n\right) V^T,$ $\mu_1\ge\cdots\ge \mu_n$ \Comment{A spectral decomposition of $Y$}
		\State $X\defeq V\textnormal{diag}\left( 1,\lambda_2,\ldots,\lambda_n\right) V^T$
		\State $Y\defeq P_{\Omega_{n}}\left( X\right) $ \Comment{Calculating by Algorithm~\ref{algo::Real_Dykstra}}
		\EndWhile
		\EndProcedure\Comment{Output: $X$ is a S.D.S matrix with $\lambda$ as its spectrum}
	\end{algorithmic}
\end{algorithm}
\section{The doubly stochastic problem}
In this section, we deal with the general case. We shall apply the alternating projection method developed in the previous section to give a solution of the DIEP. Throughout this section, $M_n\left(\mathbb{C}\right)$ will be as usual considered as a Hilbert space with the inner product
\begin{equation*}\langle A,B\rangle=\textnormal{ tr}\left( AB^*\right) =\sum_{i,j}a_{ij}\overline{b}_{ij}.
\end{equation*}
The associated norm is the Frobenius norm $\|A\|=\langle A, A\rangle^{\frac{1}{2}}.$

\begin{theorem}[Schur triangularization]
	Let $A \in M_n\left(\mathbb{C}\right)$ have eigenvalues $\mu_1, \ldots , \mu_n$ in any prescribed order. There is a unitary matrix $U \in M_n\left(\mathbb{C}\right)$ such that $U^*AU=T$ is upper triangular matrix with diagonal entries $T_{ii}=\mu_i, i=1,\ldots,n.$
\end{theorem}
We now redefine a few terms from the preceding section.
\\
Let $\lambda = \left(1,\lambda_2, \ldots , \lambda_n\right)$ be a given list of complex eigenvalues.
Define
\begin{equation*}
	\mathcal{T}=\{T\in M_n\left(\mathbb{C}\right)\colon T \textnormal{ is upper triangular with spectrum } \lambda\}.
\end{equation*}
Let $\mathcal{M}$ be the set of all $n\times n$ complex matrices with spectrum $\lambda$ i.e.
\begin{equation*}
	\mathcal{M}=\{ A\in M_n\left(\mathbb{C}\right)\colon A=UTU^* \textnormal{ for some unitary } U \textnormal{ and some } T\in \mathcal{T} \}.
\end{equation*}
Recall that  $ \textbf{N} $ be the set of nonnegative $n\times n$ matrices, defined by
\begin{equation*}
	\textbf{N}=\{A\in M_n\left(\mathbb{R}\right)\colon a_{ij}\geq 0\quad \forall  \ i,j\}.
\end{equation*}
Also recall that
\begin{equation*}
	\widehat{\Omega}_{n}=\{A\in M_n\left(\mathbb{R}\right)\colon Ae=e \textnormal{ and } A^Te=e\},
\end{equation*}
and
\begin{equation*}
	{\Omega}_{n}=\widehat{\Omega}_n\cap \textbf{N}=\{A\in M_n\left(\mathbb{R}\right) \colon Ae=e, A^Te=e \textnormal{ and } a_{ij}\geq 0\quad \forall \ i,j\}.
\end{equation*}
These leads to solve the DIEP with prescribed eigenvalues $\lambda$ via finding an $n\times n$ matrix $X$ that belongs to the intersection of $\mathcal{M}, \textbf{N}$ and $\widehat{\Omega}_{n}.$ Thus, the following problem is considered:
\begin{equation*}
	\textnormal{ Find } X\in \mathcal{M}\cap\left(\textbf{N}\cap\widehat{\Omega}_{n}\right).
\end{equation*}
Again our procedure here also depends on alternative projections by finding a best projection onto two sets $\mathcal{M}$ and $\left(\textbf{N}\cap\widehat{\Omega}_{n}\right)$ instead of the three sets $\mathcal{M}, \textbf{N}$ and $\widehat{\Omega}_n.$

We apply Boyle and Dykstra's alternating projections to find the nearest doubly stochastic matrix to a given matrix which means that the obtained matrix is the projection of a given matrix onto $\left(\textbf{N}\cap\widehat{\Omega}_{n}\right).$

Next we need the following result from \cite{Orsi}.
\begin{definition}\label{def:P_M}
	Suppose $U\in M_n\left(\mathbb{C}\right)$ is unitary and $T \in M_n\left(\mathbb{C}\right)$ is upper triangular. Let $\left\{\widehat{\lambda}_1,\ldots,\widehat{\lambda}_n\right\}$be a permutation of the list of eigenvalues $\lambda$ such that, among all possible permutations, it minimizes
	\begin{equation*} \sum_{i=1}^{n}\|\widehat{\lambda}_i-T_{ii}\|^2.
	\end{equation*}
\end{definition}
Define
\begin{equation}\label{P_M}
	P_{\mathcal{M}}\left( U,T\right) =U\widehat{T}U^*,
\end{equation}
where $\widehat{T}\in \mathcal{T}$ is given by
\begin{equation*}
	\widehat{T}_{ij}=
\cases{\widehat{\lambda}_i,  &\textnormal{if } i=j,\\
		t_{ij}, & \textnormal{otherwise}.}
\end{equation*}
\begin{lemma}[see~\cite{Orsi}]
	Let $UTU^* \in M_n\left(\mathbb{C}\right)$ be a Schur triangularization of matrix $A.$ Then $P_{\mathcal{M}}(U,T)$ satisfies
	\begin{equation*}
		\| A-P_{\mathcal{M}}(U,T)\| \leq \| A-U\tilde{T}U^*\|\quad\textnormal{ for all }\tilde{T}\in \mathcal{T}.
	\end{equation*}
\end{lemma}
Next, we shall denote by $\Re(z)$ and $\Im(z)$ to be  respectively the real part and the imaginary part of a complex number $z$.
\begin{lemma}
	Let $A\in M_n\left(\mathbb{C}\right).$ Then $\Re\left(A\right)$ is the best approximate in $M_n\left(\mathbb{R}\right)$ to $A$ in the Frobenius norm.
\end{lemma}
\begin{proof}
	Let $B\in M_n\left(\mathbb{R}\right).$ Then
	\begin{eqnarray*}
		\| A-\Re\left(A\right) \|_F^2 &= \sum_{i,j}\left| \Im\left(a_{ij}\right)\right|^2\\
		&\le\sum_{i,j}\left| \Re\left(a_{ij}\right)-b_{ij}+i\Im\left(a_{ij}\right)\right|^2\\
		&=\| A-B \|_F^2.
	\end{eqnarray*}
\end{proof}
\begin{lemma}
	Let $A \in M_n\left(\mathbb{C}\right).$ Then the $n\times n$ real matrix $A_+$ defined by
	\begin{equation*}
		\left( A_+\right)_{ij}=
\cases{	\Re\left(a_{ij}\right), &   \textnormal{ if } $\Re\left(a_{ij}\right)\ge 0$,\\
		 0, &  \textnormal{ if } $\Re\left(a_{ij}\right)<0$,}
	\end{equation*}
	for all $1 \le i, j \le n,$ is the best approximate in $\textbf{N}$ to $A$ in the Frobenius norm.
\end{lemma}
\begin{proof}
	Let $N=\left(n_{ij}\right)$ be a nonnegative matrix. Then
	\begin{eqnarray*}
		\| A-A_+ \|_F^2&=\sum_{\Re\left(a_{ij}\right)<0}\left|\Re\left(a_{ij}\right)\right|^2+ \sum_{i,j}\left| \Im\left(a_{ij}\right)\right|^2\\
		&\le\sum_{\Re\left(a_{ij}\right)<0}\left|\Re\left(a_{ij}\right)-n_{ij}\right|^2+ \sum_{i,j}\left| \Im\left(a_{ij}\right)\right|^2\\
		&\le \sum_{i,j}\left|\Re\left(a_{ij}\right)-n_{ij}\right|^2+\sum_{i,j}\left| \Im\left(a_{ij}\right)\right|^2\\
		&=\| A-N \|_F^2.
	\end{eqnarray*}
	This completes the proof.
\end{proof}
The following proposition will play an important role in our study.
\begin{proposition}
	Let $A \in M_n\left(\mathbb{C}\right).$ Then the best approximate in ${\Omega}_{n}$ to $A$ in the Frobenius norm is the projection of $\Re\left(A\right)$ onto $\Omega_{n}$ which can be found by applying Theorem~\ref{th::dykstra} so that $C=\Omega_{n}=\widehat{\Omega}_{n}\cap\textbf{N}, r=2,$ and $x_0^1=\Re\left( A\right),$ where $\Re\left( A\right) =\left( \Re\left( a_{ij}\right) \right) _{ij}.$
	Thus $x_1^1=B_{\Re\left( A\right) }, x_2^1=\left( B_{\Re\left( A\right) }\right) _+,$ and so on.
\end{proposition}
\begin{proof}
	Let $A\in M_n\left(\mathbb{C}\right),$ $\Re\left( A\right) =P_{M_n\left(\mathbb{R}\right)}\left( A\right)$ and $B=P_{\Omega_{n}}\left( \Re\left( A\right) \right).$
	In order to prove that  $B=P_{\Omega_{n}}\left( A\right),$ we need to show that \begin{equation*}
		\| A-B\|\le \| A-C\|
	\end{equation*}
	for all $C\in \Omega_{n}.$
	
	\begin{eqnarray*}
		\|A-B\|^2&=\|A-\Re\left( A\right)\|^2+\|\Re\left( A\right)-B\|^2-2\langle A-\Re\left( A\right),\Re\left( A\right)- B\rangle\\
		&=\|A-\Re\left( A\right)\|^2+\|\Re\left( A\right)-B\|^2 \quad\left(\textnormal{Since } M_n\left(\mathbb{R}\right) \textnormal{ is a closed linear space}.\right)\\
		&\le \|A-\Re\left( A\right)\|^2+\|\Re\left( A\right)-C\|^2\quad \textnormal{ for all } C\in \Omega_{n}\\
		&= \|A-\Re\left( A\right)\|^2+\|\Re\left( A\right)-C\|^2-2\underbrace{\langle A-\Re\left( A\right),\Re\left( A\right)- C\rangle}_{0}\\
		&=\|A-C\|^2.
	\end{eqnarray*}
	This will complete the proof.
\end{proof}
Here is a proposed algorithm for solving the DIEP given in the following.

\begin{algorithm}[h]
	\caption{Algorithm for finding D.S matrix with prescribed spectrum}
	\begin{algorithmic}[1]
		\Require A nonnegative $A$ and $\mathrm{Spec}\left(A\right) =\lambda$
		\Comment{Due to \cite[p. 198]{Orsi} }
		\Procedure {Return}{$D.S$}
		\State $Y\defeq A$ 	
		\While {$\|\Re\left(X\right)-Y\|\ge \varepsilon$}
		\State $Y\defeq UTU^*$\Comment{A Schur decomposition of $Y$}
		\State $X\defeq P_\mathcal{M}\left( U,T\right)$ \Comment{$P_\mathcal{M}\left( U, T\right)$is given by Definition~\ref{def:P_M}}
		\State $Y\defeq P_{\Omega_{n}}\left( \Re\left( X\right) \right) $ \Comment{Calculating by Dykstra's algorithm}
		\EndWhile
		\EndProcedure\Comment{Output: $Y$ is a D.S matrix with $\lambda$ as its spectrum}
	\end{algorithmic}
\end{algorithm}
\section{Convergence}
This section concludes with a discussion that describes the conditions under which SDIEP and DIEP algorithms converge.

The following theorem illustrates one important use of \cite{Converge_MAP,Convergence_of_MAP}. Another good reference is \cite{Zarantonello_71} and the references therein.
\begin{theorem}\label{th:conv_compact}
	Let $C_1$ be a compact set in a finite dimensional Hilbert space $H$ and let $C_2$ be a closed (nonempty) set in $H.$ Suppose that an algorithm generates two sequences $\{x_i\}_i$ and $\{y_i\}_i$ defined by
	\begin{equation*}
		y_i=P_{C_2}\left( x_i\right) \quad\textnormal{ and }\quad x_{i+1}=P_{C_1}\left( y_{i}\right),
	\end{equation*}
	for all $i=0,1,\ldots\cdot$
	If, in addition, $P_{C_2}$ is a continuous map, then
	\begin{equation*}
		\|x-y\|=\lim\limits_{i\to\infty}\|x_i-y_i\|,
	\end{equation*}
	for some $x\in C_1$ and its projection $y\in C_2.$
\end{theorem}
\begin{proof}
	Since $C_1$ is compact and $x_i \in C_1,$ it follows that $ \left\{x_i\right\}_i$ has at least one accumulation point in $C_1.$ Moreover, it has a convergent subsequence $\left\{x_{i_k}\right\}$ such that
	\begin{equation}\label{eq:4.1}
		\lim\limits_{k\to\infty}x_{i_k}=x,
	\end{equation}
	for some $x\in C_1.$ Relying on continuity of projection $P_{C_2},$ it is possible to say that
	\begin{eqnarray}\label{eq:4.2}
		\lim\limits_{k\to\infty}P_{C_2} \left(x_{i_k}\right) &= P_{C_2} \left(x\right)\nonumber\\
		\lim\limits_{k\to\infty} y_{i_k}&=y,
	\end{eqnarray}
	where $y=P_{C_2}\left( x\right)\in C_2.$
	
	Proposition~\ref{Prop:Reduc_Dist} implies $\lim\limits_{i\to\infty}\|x_{i}-y_{i}\|$ exists. Then, the equalities~\ref{eq:4.1} and \ref{eq:4.2} imply
	\begin{equation*}
		\lim\limits_{i\to\infty}\| x_{i}-y_{i}\|=\lim\limits_{k\to\infty}\| x_{i_k}-y_{i_k}\|=\| x-y\|.
	\end{equation*}
	This completes the proof of the theorem.
\end{proof}
To convert Theorem~\ref{th:conv_compact} to the context of the DIEP, we need to establish that projection operators on convex sets have the following crucial properties.
\begin{proposition}[see~\cite{Zarantonello_71}]\label{Proj_Continuous}
	If $C$ is a closed and convex set in a Hilbert space, then the projection operator $P_C\colon H \mapsto C$ is non-expansive; i.e., for all $x, y \in H$
	\begin{equation*}
		\|P_C(y)-P_C(x)\|\le \|y-x\|.
	\end{equation*}
	Furthermore, $P_C$ is a continuous map.
\end{proposition}
Next, we have the following result.
\begin{corollary}
	Suppose that the $\left\{x_i\right\}_i$ is a sequence generated by alternating projections onto the compact set $\mathcal{M}.$ Then there is an accumulation point $x$ of the sequence $\left\{x_i\right\}_i$ satisfying
	\begin{equation}
		\|x-P_{\Omega_{n}}\left( x\right) \|=\lim\limits_{i\to \infty} \|x_i-P_{\Omega_{n}}\left( x_i\right)\|.
	\end{equation}
\end{corollary}
\begin{proof}
	It suffices to show that $P_{\Omega_{n}}$ is a continuous as a linear operator. Since $\Omega_{n}$ is a closed convex set, this implies that $P_{\Omega_{n}}$ is continuous by Proposition~\ref{Proj_Continuous}. Therefore, the result follows immediately from Theorem~\ref{th:conv_compact}.
\end{proof}
The above corollary asserts that every accumulation point of $\left\{x_i\right\}_i$ is a solution of the DIEP (SDIEP) with prescribed eigenvalues whenever the limit vanishes.

\subsection*{Semi-algebraic sets}
A short summary of semi-algebraic sets with definitions and most of the properties that we will need can be found in \cite[Section\,1]{Robert_1980}.
A subset $\mathcal{A}$ of $\mathbb{R}^n$ is semi-algebraic if there exist a finite number of polynomials $f_{i,j}, g_{i,j}$ in $n$ variables with real coefficients for $i\in\{1,\ldots,I\}, j\in\{1,\dots,J\},$ so that
\begin{equation*}
	\mathcal{A}=\cup_{i=1}^{I}\cap_{j=1}^{J}\{x\in \mathbb{R}^n\colon f_{i,j}\left( x\right) >0, g_{i,j}\left( x\right) =0\}.
\end{equation*}
In other words, a semi-algebraic set is defined by finitely many systems of equations and/or inequalities of polynomials.
A semi-algebraic set $\mathcal{A}\subset\mathbb{R}^n$ is called bounded if it is contained in a ball of finite radius. The following theorem illustrates one important use of semi-algebraic sets.
\begin{theorem}[Semi-algebraic intersections, \cite{Drusvyatskiy_2015}] \label{th:Semi-alg-Alt_proj}
	Consider two nonempty closed semi-algebraic subsets $X, Y$ of an Euclidean space $E$ such that $X$ is bounded. If the method of alternating projections starts in $Y$ and near $X,$ then the distance of the iterates to the intersection $X \cap Y$ converges to zero, and hence every limit point lies in $X \cap Y.$
\end{theorem}
The following corollary will play an important role in our study of the convergence of both SDIEP and DIEP algorithms.
\begin{corollary}
	Let $\left\{x_i\right\}_i$ and $\left\{y_i\right\}_i$ be the alternating projection sequences defined iteratively as follows:
	\begin{equation*}
		x_{i+1} = P_{\Omega_{n}} (y_i), y_{i+1} = P_{\mathcal{M}} (x_{i+1}).
	\end{equation*}
	Then $\|x_{i}-y_{i}\|$ converges to zero, and hence SDIEP and DIEP algorithms converge.
\end{corollary}
\begin{proof}
	By Theorem~\ref{th:Semi-alg-Alt_proj}, it suffices to prove that $\mathcal{M}$ and $\Omega_{n}$ are two semi-algebraic sets and $\Omega_{n}$ bounded. For this purpose, we will use the Newton identities, to write the set $\mathcal{M}$ as
	\begin{equation*}
		\mathcal{M}=\left\{A\in M_n\left(\mathbb{C}\right)\equiv\mathbb{R}^{2n^2} \colon
		0=\sum_{k=0}^{n}(-1)^k\sigma_k\left(\lambda\right)A^{n-k}\right\}
	\end{equation*}
	so that $\mathcal{M}$ is a semi-algebraic set. Moreover, there is a natural way to write $\Omega_{n}$ as
	\begin{equation*}
\hspace{-2cm} \Omega_{n}=\left\{A=\left(a_1,\ldots, a_n,a_{n+1},\ldots, a_{n^2}\right)^T\in M_n\left(\mathbb{R}_{+}\right)\colon \sum_{t=i}^{n}a_t-1=0,\  i=1\textnormal{ mod } n\right\}
	\end{equation*}
	which shows that $\Omega_{n}$ is also a semi-algebraic set. As $\Omega_{n}$ is  bounded, then the proof is complete.
\end{proof}
Next, we shall show the linear convergence of our proposed algorithms. But first, we  begin with the following.
\begin{theorem}[Generic transversality, \cite{Drusvyatskiy_2015}]\label{th:Semi-alg_transver}
	Consider two closed semi-algebraic subsets $M, N$ of a Euclidean space $E.$ Then, for almost every vector $v$ in $E,$ transversality holds at every point in the (possibly empty) intersection of the sets $M$ and $N-v.$
\end{theorem}
\begin{theorem}[Linear Convergence, \cite{Lewis_2008}]\label{th:Linear_Conv_manifolds}
	In the Euclidean space $E,$ let $M$ and $N$ be two transverse manifolds around point $\overline{x}\in M\cap N.$ If the initial point $x_0\in E$ is close to $\overline{x},$ then the method of alternating projections \begin{equation*}
		x_{k+1}=P_MP_N\left(x_k\right)\quad \left(k=0,1,2\ldots\right)
	\end{equation*}is well defined, and the distance $d_{M\cap N}\left(x_k\right)$ from the iterate $x_k$ to the intersection $M\cap N$ decreases $Q-$linearly to zero. More precisely, given any constant $c$ strictly larger than the cosine of the angle of intersection the manifolds, $c\left(M,N,\overline{x}\right),$ if $x_0$ is close to $\overline{x},$ then the iterates
	\begin{equation*}
		d_{M\cap N}\left(x_{k+1}\right)\le c\cdot d_{M\cap N}\left(x_k\right) \quad \left(k=0,1,2\ldots\right).
	\end{equation*}
	Furthermore, $x_k$ converges linearly to some point $x^*\in M\cap N\colon$ for some constant $\alpha>0,$
	\begin{equation*}
		\|x_k-x^*\|\le \alpha c^k\quad \left(k=0,1,2\ldots\right).
	\end{equation*}
\end{theorem}
As a conclusion, we have  the following result.
\begin{corollary}
	Let $\left\{x_i\right\}_i$ and $\left\{y_i\right\}_i$ be the alternating projection sequences defined iteratively as follows:
	\begin{equation*}
		x_{i+1} = P_{\Omega_{n}} (y_i), y_{i+1} = P_{\mathcal{M}} (x_{i+1}).
	\end{equation*}
	Then $x_{i}$ converges linearly to some point $x^*\in \mathcal{M}\cap \Omega_{n}$ with rate $c\in]0,1[,$ and hence SDIEP and DIEP algorithms converge.
\end{corollary}
\begin{proof}
	By Theorem~\ref{th:Semi-alg_transver} and Theorem~\ref{th:Linear_Conv_manifolds}, the proof is straightforward by taking $M=\mathcal{M}$ and $N=\Omega_{n}.$
\end{proof}

Notice that from a computational perspective, we learned how to compute the projections onto the sets $\mathcal{M}$ and $\Omega_{n}$ in the previous sections, and so an intersection point of $\mathcal{M}\cap{\Omega}_n$ can be viewed as the common fixed point of the projection operator for each of the sets. An equivalent statement can be achieved by the following theorem.
\begin{theorem}\label{th:fixedPt}
	A point $x$ is an intersection point of the sets $\mathcal{M}$ and ${\Omega}_n$ if and only if $x$ is a common fixed point of all the projection operators for those sets, i.e.,
	\begin{equation*}
		x=P_{\mathcal{M}}\left(x\right)=P_{\Omega_{n}}\left(x\right).
	\end{equation*}
\end{theorem}
The following corollary is an immediate consequence the above theorem.
\begin{corollary}
	Let $x$ be an intersection point of the sets $\mathcal{M}$ and ${\Omega}_n.$ Then $x$ is a fixed point of the combined projection operator $P_{\Omega_{n}}P_{\mathcal{M}},$ i.e.,
	\begin{equation*}
		x=P_{\Omega_{n}}P_{\mathcal{M}}\left(x\right).
	\end{equation*}
\end{corollary}
It is worthy to note that the above corollary shows that a fixed point of the combined projection operator is a necessary condition for the intersection point of the sets $\mathcal{M}$ and $\Omega_{n},$ but not vice versa. For more on this issue,  see \cite[Section~5]{Orsi}.
\section{Numerical experiments}
We have tested both SDIEP and DIEP on various dimensions and our methods look promising. For demonstration purpose, we only consider the cases when $n=4$ and $7.$ The following calculations were performed after implementing our algorithms via Maple 18. For illustration purposes, the detailed analysis of the result in the first case as well as the Maple codes can be found in the appendices below.
\subsection{For $n=4$}
\subsubsection{Real Spectrum}	For the list  $\lambda = \left(1,\frac{3}{4},-\frac{1}{4},-\frac{1}{2}\right)$,
we obtained  the realizing matrix
$\small{ M=\left(\begin{array}{cccc}
0.101196761727949&0.135883686457413&0.359487108073623&0.403432443741015\\
0.135883686457413&0.804884261260186&0.0371665534062483&0.0220654988761529\\
0.359487108073623&0.0371665534062483&0.0613816290745809&0.541964709445548\\
0.403432443741015&0.0220654988761529&0.541964709445548&0.0325373479372842
		\end{array}\right)}$
	which gives a numerical solution of SDIEP.
	\subsubsection{Complex Spectrum}	Similarly, for the given list of complex numbers $\lambda = \left(1, \frac{1}{3},\frac{1}{6}+\frac{\sqrt{519\times 10^{-5}}}{6}i,\frac{1}{6}-\frac{\sqrt{519\times 10^{-5}}}{6}i\right),$
	we found the following realizing matrix\\
$\small{
M=\left(\begin{array}{cccc}
			0.0648048716687098& 0.935195128331290&
			0& 0\\  0.340130351957920& 0.0640767670903937&
			0.294236123300947& 0.301556757650739\\
			0.304564892725389& 0.000427653198872363& 0.433841286483687&
			0.261166167592051\\  0.290499883647981&
			0.000300451379443567& 0.271922590215366& 0.437277074757210
			\end{array}\right)}
$
		which gives a numerical solution to DIEP.
		\subsection{For $n=7$}
		For the list  $\lambda = \left(1, 0.85, 0.75, 0.5,-0.25,-0.35, -0.5\right),$  we have found the following realizing matrix
	\begin{equation*}
\hspace*{-1cm}		M=\left(\begin{tabular}{ccc}
		$M_1$&$M_2$&$M_3$\\$M_4$&$M_5$&$M_6$
		\end{tabular}\right), \textnormal{ where}
	\end{equation*}
$$
\hspace{-2cm}
	  M_1=\left(\begin{array}{ccc}
				0.867163231154133& 0.0385836434072440& 0.0226163301545610\\
				0.0385836434072441&0.782190028144753& 0.0354563052552376\\
				0.0226163301545610& 0.0354563052552377&0.349473772544678
			\end{array}\right) \in M_3\left(\mathbb{R}\right), $$
$$ \hspace{-2cm}
			M_2=\left(\begin{array}{cc}
				0.0374533710504067& 0.0220365385021468\\
				0.0117398898208207&0.0397490544343596\\
				0.0832079216768592& 0.00664932302701010
			\end{array}\right) \in M_{3,2}\left(\mathbb{R}\right),$$
 $$\hspace{-2cm}
	M_3=\left(\begin{array}{cccc}
	 0&0.0121468857315085\\
	0.0116241884359759&0.0806568905016094\\
0.478117947670726&0.0244783996709279
\end{array}\right) \in M_{3,2}\left(\mathbb{R}\right),$$

		$$ \hspace{-1cm}	M_4=\left(\begin{array}{cccc}
				0.0374533710504067& 0.0117398898208207& 0.0832079216768592\\
				0.0220365385021468& 0.0397490544343597&0.00664932302701007\\
				0&0.0116241884359759& 0.478117947670726\\
				0.0121468857315085& 0.0806568905016095& 0.0244783996709279
			\end{array}\right)\!\in M_{4,3}\left(\mathbb{R}\right),$$
		$$  M_5= \left(\begin{array}{cc}
				0&0.371565767122184\\
				0.371565767122184& 0.0011729681564368\\
				0.100748927207363&0.240450792235105\\
				0.395284123122367& 0.318375556522757
			\end{array}\right) \in M_{4,2}\left(\mathbb{R}\right),$$
		$$	M_6= \left(\begin{array}{cc}
		 0.100748927207363& 0.395284123122367\\
		0.240450792235105& 0.318375556522757\\
	 0& 0.169058144450830\\
	 0.169058144450830& 0
		\end{array}\right) \in M_{4,2}\left(\mathbb{R}\right).$$

		\section*{Declaration of Competing Interest}
		The authors declare that they have no competing interests.
		
\appendix
\section{Nonnegative Matrix with Eigenvalues $\lambda=\left(1, \frac{3}{4}, -\frac{1}{4}, -\frac{1}{2}\right)$}	
with(LinearAlgebra): with(ListTools): with(combinat):
\\
MakeOh$\,\defeq$ \textbf{proc}$(x)$ \\
\textbf{if} $x < 0$ \textbf{then} $x-x$\\
\textbf{elif} $x>0$ \textbf{then} $x$\\
\textbf{else} $x-x$\\
\textbf{end if}\\
\textbf{end proc};	\\
$A\defeq$ RandomMatrix(4, 4, generator $= 0\,.. \,99$); \#Starting Point
\begin{equation*}
A\defeq\left[\begin{array}{cccc}
64&70& 85& 9\\
52& 10& 6& 19\\
23& 28& 0& 31\\
61& 22& 66& 0
\end{array}\right]
\end{equation*}
$Y\defeq \frac{1}{2}\cdot\left(A + Transpose(A)\right):$\\
$E, V\defeq$ Eigenvectors$(Y):$\\
$X\defeq V\cdot$DiagonalMatrix$\left(\left[1, \frac{3}{4}, -\frac{1}{4}, -\frac{1}{2}\right], 4, 4\right)\cdot$MatrixInverse(evalf$(V, 100)$):\\
$X\defeq$ simplify$(\frac{1}{2}\cdot($evalf$(X, 100)+$Transpose(evalf$(X, 100))))$;\\
\begin{equation*}
X\defeq\left[\begin{array}{cccc}
0.4684022410& 0.5592546853& 0.2687477422& 0.08977571410\\
0.5592546853& 0.2084922562& -0.09638457078& -0.07803895284\\
0.2687477422&-0.09638457078& 0.1212531543& 0.6082030224\\
0.08977571410& -0.07803895284& 0.6082030224& 0.2018523484
\end{array}\right]
\end{equation*}
$Y\defeq$ map(MakeOh, map(Re, $X$));\\
\begin{equation*}
Y\defeq\left[\begin{array}{cccc}
		0.4684022410& 0.5592546853& 0.2687477422& 0.08977571410\\
		0.5592546853& 0.2084922562& 0.& 0.\\
		0.2687477422& 0.& 0.1212531543& 0.6082030224\\
		0.08977571410& 0.& 0.6082030224& 0.2018523484
	\end{array}\right]
\end{equation*}
simplify(MatrixNorm$(X - Y,$ Frobenius));
\begin{equation*}
	0.1753856530
\end{equation*}
\textbf{while} simplify(MatrixNorm$(X - Y,$ Frobenius))$>10^{-9}$ \textbf{do}\\
$E, V\defeq$ Eigenvectors$(Y)$:\\
$X\defeq V\cdot$DiagonalMatrix$\left(\left[1, \frac{3}{4}, -\frac{1}{4}, -\frac{1}{2}\right], 4, 4\right)\cdot$MatrixInverse$(V)$:\\
X $\defeq$ simplify($\frac{1}{2}\cdot$($X +$ Transpose$(X)$)):\\
$Y\defeq$ map(MakeOh, map(Re, $X$)):\\
\textbf{end do}:\\
 simplify(MatrixNorm$(X - Y,$ Frobenius));
 \begin{equation*}
 1.132499970\times10^{-10}
 \end{equation*}
 $Y$;
\begin{equation*}
\left[\begin{array}{cccc}
0.2062464680& 0.1778473020& 0.3712215523& 0.4324230299\\
0.1778473020& 0.7937535321& 0& 0\\
0.3712215523& 0& 0& 0.4947565885\\
0.4324230299& 0& 0.4947565885& 0
\end{array}\right]
\end{equation*}
Eigenvalues$(Y);$
\begin{equation*}
\left[\begin{array}{c}
1.0000000001027831\\ 0.749999999935391\\ -0.2500000000352335\\ -0.49999999990294086\end{array}\right]
\end{equation*}
\section{Symmetric Doubly Stochastic Matrix with Eigenvalues $\lambda=\left(1, \frac{3}{4}, -\frac{1}{4}, -\frac{1}{2}\right)$}
with(LinearAlgebra): with(ListTools): with(combinat):
\\
MakeOh$\,\defeq$ \textbf{proc}$(x)$ \\
\textbf{if} $x < 0$ \textbf{then} $x-x$\\
\textbf{elif} $x>0$ \textbf{then} $x$\\
\textbf{else} $x-x$\\
\textbf{end if}\\
\textbf{end proc};	
\begin{equation*}
A\defeq\left[\begin{array}{cccc}
	0.2062464680& 0.1778473020& 0.3712215523& 0.4324230299\\
	0.1778473020& 0.7937535321& 0& 0\\
	0.3712215523& 0& 0& 0.4947565885\\
	0.4324230299& 0& 0.4947565885& 0
\end{array}\right]
\end{equation*}
lambda$\,\defeq \left[1, \frac{3}{4}, -\frac{1}{4}, -\frac{1}{2}\right];$
\begin{equation*}
\lambda\defeq\left[1, \frac{3}{4}, -\frac{1}{4}, -\frac{1}{2}\right]
\end{equation*}
$Y\defeq \frac{1}{2}\cdot\left(A + A^{\%T}\right);$
\begin{equation*}
\hspace*{-2cm}	Y\defeq\left[\begin{array}{cccc}
0.206246468000000&0.177847302000000&0.371221552300000&0.432423029900000\\
0.177847302000000&0.793753532100000&0.&0.\\
0.371221552300000&0.&0.&0.494756588500000\\
0.432423029900000&0.&0.494756588500000&0.
	\end{array}\right]
\end{equation*}
$T, U\defeq$ Eigenvectors$(Y)$;
\begin{equation*}
\hspace*{-3cm}\begin{array}{cc}
&T, U\defeq \left[\begin{array}{c}
1.00000000010278\\ 0.749999999935391\\-0.250000000035233\\-0.499999999902941
\end{array}\right],\\
&\left[\begin{array}{cccc}
0.576445412257448& 0.210911876437704& 0.780276165192027& -0.119983220068859\\
0.497071597458245& -0.857304686695529& -0.132952853831712& 0.0164936299277058\\
0.446648531011842& 0.326847445760011& -0.518540523285432& -0.651760356769258\\
0.470250575153941& 0.337217439862715& -0.323432172799954& 0.748692476533406
\end{array}\right]\end{array}
\end{equation*}
$W\defeq$ IdentityMatrix$(4)-$ Matrix$\left(\left[\left[\frac{1}{4}, \frac{1}{4}, \frac{1}{4}, \frac{1}{4}\right], \left[\frac{1}{4}, \frac{1}{4}, \frac{1}{4}, \frac{1}{4}\right], \left[\frac{1}{4}, \frac{1}{4}, \frac{1}{4}, \frac{1}{4}\right], \left[\frac{1}{4}, \frac{1}{4}, \frac{1}{4}, \frac{1}{4}\right]\right]\right)$;
\begin{equation*}
W\defeq\left[\begin{array}{cccc}
\frac{3}{4}& -\frac{1}{4}& -\frac{1}{4}& -\frac{1}{4}\\
-\frac{1}{4}& \frac{3}{4}& -\frac{1}{4}& -\frac{1}{4}\\
-\frac{1}{4}& -\frac{1}{4}& \frac{3}{4}& -\frac{1}{4}\\
-\frac{1}{4}& -\frac{1}{4}& -\frac{1}{4}& \frac{3}{4}
\end{array}\right]
\end{equation*}
Proj$_{-}$Omega$\,\defeq$ \textbf{proc}$(M$::Matrix$(4, 4), n$::nonnegint)\\
\textbf{local} $X\defeq$ \textbf{proc}$(n$::nonnegint)\\
\textbf{local} $W\defeq$ IdentityMatrix$(4)-$ Matrix$\left(\left[\left[\frac{1}{4}, \frac{1}{4}, \frac{1}{4}, \frac{1}{4}\right], \left[\frac{1}{4}, \frac{1}{4}, \frac{1}{4}, \frac{1}{4}\right], \left[\frac{1}{4}, \frac{1}{4}, \frac{1}{4}, \frac{1}{4}\right], \left[\frac{1}{4}, \frac{1}{4}, \frac{1}{4}, \frac{1}{4}\right]\right]\right)$;\\
\textbf{if} $n = 0$ \textbf{then} \\
\textbf{return} $M;$\\
\textbf{else}\\
\textbf{return} $W\cdot Y(n-1)\cdot W +$ Matrix$\left(\left[\left[\frac{1}{4}, \frac{1}{4}, \frac{1}{4}, \frac{1}{4}\right], \left[\frac{1}{4}, \frac{1}{4}, \frac{1}{4}, \frac{1}{4}\right], \left[\frac{1}{4}, \frac{1}{4}, \frac{1}{4}, \frac{1}{4}\right], \left[\frac{1}{4}, \frac{1}{4}, \frac{1}{4}, \frac{1}{4}\right]\right]\right)$;\\
\textbf{end if};\\
\textbf{end proc};\\
\textbf{local} $Y\defeq$ \textbf{proc}$(n$::nonnegint)\\
\textbf{if} $n = 0$ \textbf{then}\\
\textbf{return} $M;$\\
\textbf{else}\\
\textbf{return} map(MakeOh, $X(n)-$ Id$_{-}$two$(n - 1, X(n - 1), Y(n - 1)))$;\\
\textbf{end if};\\
\textbf{end proc};\\
\textbf{local} Id$_{-}$one$\defeq$ \textbf{proc}$(n$::nonnegint, $Q$::Matrix$(4, 4)$, $R$::Matrix$(4, 4)$)\\
\textbf{if} $n = 0$ \textbf{then}\\
\textbf{return} $0$;\\
\textbf{else}\\
\textbf{return} $Q - R +$ Id$_{-}$one$(n - 1, X(n - 1), Y(n - 1))$;\\
\textbf{end if};\\
\textbf{end proc};\\
Id$_{-}$two$\defeq$ \textbf{proc}$(n$::nonnegint,$Q$::Matrix$(4, 4)$, $R$::Matrix$(4, 4)$)\\
\textbf{if} $n = 0$ \textbf{then}\\
\textbf{return} $0$;\\
\textbf{else}\\
\textbf{return} $R - Q +$ Id$_{-}$two$(n - 1, X(n - 1), Y(n - 1))$;\\
\textbf{end if};\\
\textbf{end proc};\\
\textbf{return} $Y(n), X(n);$\\
\textbf{end proc}\\
ProjOn$_{-}M\defeq$ \textbf{proc}$(T$::Matrix$(4, 4)$)::Matrix$(4, 4)$;\\
\textbf{global} $N\defeq$Matrix$(4,4)$;\\
\textbf{local} $l, m, k,$\\
\textbf{local} $L\defeq$ Vector$(4!)$, $i\defeq 0$, $j\defeq 0$, $P\defeq$ [seq$(0, i = 1\,..\,4)$];\\
\textbf{local} $H\defeq T$;\\
$P\defeq$ permute([\textquotesingle\$\textquotesingle$(1\,.. \,4)$]);\\
\textbf{for $j$ from $1$ by $1$ to $4!$ do}\\
$L[j]\defeq L[j] +$ add(abs(lambda$[i]- T[P[j, i], P[j, i]])^2, i = 1\,..\,4$);\\
\textbf{end do};\\
$j\defeq$ Search($\min(L), L$);\\
\textbf{for $i$ from $1$ by $1$ to $4$ do}\\
$N[P[j, i], P[j, i]]\defeq$ lambda$[i]$;\\
\textbf{end do};\\
\textbf{for $l$ from $1$ to $4$ do}\\
\textbf{for $m$ from $l + 1$ to $4$ do} \\
$N[l, m]\defeq T[l, m]$;\\
\textbf{end do}; \\
\textbf{end do};\\
\textbf{for $l$ from $1$ to $4$ do}\\
\textbf{for $k$ from $l - 1$ by $-1$ to $1$ do}\\
$N[l, k] \defeq T[l, k]$;\\
\textbf{end do};\\
\textbf{end do}; \\
\textbf{return} $N$;\\
\textbf{end proc}\\
$X\defeq U\cdot$DiagonalMatrix([lambda], $4, 4)\cdot U^*$;
\begin{equation*}
\hspace*{-2.8cm}X\defeq\left[\begin{array}{cccc}
0.206246467988774& 0.177847301955404& 0.371221552256145& 0.432423029876560\\
0.177847301955404& 0.793753532122686&  -0.3745155187\times10^{-10}& -0.4238765232\times10^{-10} \\
0.371221552256145&-0.3745154320\times10^{-10}& -0.4535897046\times10^{-10}& 0.494756588538803\\
 0.432423029876560&-0.4238763150\times10^{-10}& 0.494756588538803&-0.6610203878\times10^{-10}
\end{array}\right]
\end{equation*}
$B\defeq W\cdot$map(Re, $X$)$\cdot W+$Matrix$\left(\left[\left[\frac{1}{4}, \frac{1}{4}, \frac{1}{4}, \frac{1}{4}\right], \left[\frac{1}{4}, \frac{1}{4}, \frac{1}{4}, \frac{1}{4}\right], \left[\frac{1}{4}, \frac{1}{4}, \frac{1}{4}, \frac{1}{4}\right], \left[\frac{1}{4}, \frac{1}{4}, \frac{1}{4}, \frac{1}{4}\right]\right]\right)$;
\begin{equation*}
\hspace*{-2cm}B\defeq\left[\begin{array}{cccc}
0.109408351018717& 0.135043564505004& 0.354823488127274& 0.400724596349005\\
0.135043564505004& 0.804984174191945& 0.0376363153533354& 0.0223359459497152\\
0.354823488127273& 0.0376363153533354& 0.0640419886669564& 0.543498207852435\\
0.400724596349005& 0.0223359459497152& 0.543498207852435&0.0334412498488450
\end{array}\right]
\end{equation*}
$n\defeq 1;$\\
\begin{equation*}
	n\defeq 1
\end{equation*}
$M\defeq$ map(MakeOh, map(Re, $B$));
\begin{equation*}
	\hspace*{-2cm}M\defeq\left[\begin{array}{cccc}
		0.109408351018717& 0.135043564505004& 0.354823488127274& 0.400724596349005\\
		0.135043564505004& 0.804984174191945& 0.0376363153533354& 0.0223359459497152\\
		0.354823488127273& 0.0376363153533354& 0.0640419886669564& 0.543498207852435\\
		0.400724596349005& 0.0223359459497152& 0.543498207852435&0.0334412498488450
	\end{array}\right]
\end{equation*}
\textbf{while} simplify(MatrixNorm($X-B$,Frobenius))$>10^{-14} $ \textbf{do}\\
$T,U\defeq\;$SchurForm($B,$output$=[$\textquotesingle $T$\textquotesingle,\textquotesingle$Z$\textquotesingle$]$);\\ $X\defeq U\cdot$ProjOn$_{-}M(T)\cdot U^*$;\\
\textbf{while} simplify(MatrixNorm($M - B,$ Frobenius))$>10^{-14}$ \textbf{do}\\
\textbf{while} simplify(MatrixNorm(Proj$_{-}$Omega$(M, n)[1]-$Proj$_{-}$Omega$(M, n)[2],$ Frobenius)) $>10^{-14}$ \textbf{do}\\
$n\defeq n + 1$;\\
$B\defeq$ Proj$_{-}$Omega$(M, n)[1]$;\\
\textbf{end do};\\
$M\defeq$ map(MakeOh, map(Re, $B$));\\
$n\defeq 1$;\\
\textbf{end do}: \\
\textbf{end do}: \\
B;
\begin{equation*}
\hspace*{-2.1cm}\left[\begin{array}{cccc}
0.101196761727949&0.135883686457413&0.359487108073623&0.403432443741015\\
0.135883686457413&0.804884261260186&0.0371665534062483&0.0220654988761529\\
0.359487108073623&0.0371665534062483&0.0613816290745809&0.541964709445548\\
0.403432443741015&0.0220654988761529&0.541964709445548&0.0325373479372842
\end{array}\right]
\end{equation*}
simplify(MatrixNorm$(X - B,$ Frobenius));
\begin{equation*}
8.50191497400929368\times 10^{-16}
\end{equation*}
Eigenvalues$(B);$
\begin{equation*}
	\left[\begin{array}{c}
		1.00000000000000\\ 0.750000000000000\\ -0.250000000000000\\ -0.500000000000000\end{array}\right]
\end{equation*}
$B[1, 1] + B[1, 2] + B[1, 3] + B[1, 4];$
\begin{equation*}
	1
\end{equation*}
$B[2, 1] + B[2, 2] + B[2, 3] + B[2, 4];$
\begin{equation*}
	1
\end{equation*}
$B[3, 1] + B[3, 2] + B[3, 3] + B[3, 4];$
\begin{equation*}
	1
\end{equation*}
$B[4, 1] + B[4, 2] + B[4, 3] + B[4, 4];$
\begin{equation*}
	1
\end{equation*}
$B[1, 1] + B[2, 1] + B[3, 1] + B[4, 1];$
\begin{equation*}
	1
\end{equation*}
$B[1, 2] + B[2, 2] + B[3, 2] + B[4, 2];$
\begin{equation*}
	1
\end{equation*}
$B[1, 3] + B[2, 3] + B[3, 3] + B[4, 3];$
\begin{equation*}
	1
\end{equation*}
$B[1, 4] + B[2, 4] + B[3, 4] + B[4, 4];$
\begin{equation*}
	1
\end{equation*}

\section*{References}
		\begin{thebibliography}{99}
			\bibitem{Alberti} Alberti PM, Uhlmann A. Stochasticity and Partial Order: Doubly Stochastic Maps and Unitary Mixing, Reidel, Boston 1982.
			\bibitem{Atrianfar} Atrianfar H, Haeri A. Average Consensus in Networks of Dynamic Multi-agents with Switching Topology: Infinite Matrix Products, ISA Transactions, 2012; 51(4):522-530.
			\bibitem{Bauschke_2013}Bauschke HH, Luke DR,  Phan HM et al. Restricted Normal Cones and the Method of Alternating Projections: Theory. Set-Valued Var. Anal 2013; 21:431-473.
			\bibitem{dykstra}Boyle JP, Dykstra RL. A method for finding projections onto the intersections of convex sets in Hilbert spaces. In Advances in Order Restricted Statistical Inference, Lecture Notes in Statistics 1986; 37:28-47.
			\bibitem{MAP}Br\`egman LM. The method of successive projection for finding a common point of convex sets, Soviet Mathematics 1965; 6:688-692.
			\bibitem{Brockett} Brockett RW. Dynamical systems that sort lists, diagonalize matrices, and solve linear programming problems, Linear Algebra Appl.  1991; 146:79-91.
			
			\bibitem{Bullo} Bullo F, Cortes J, Martinez S. Distributed Control of Robotic Networks, Applied Mathematics Series, Princeton University Press; 2009.
			\bibitem{Cheney_59}Cheney W, Goldstein A. Proximity maps for convex sets, Proc. Amer. Math. Soc. 1959; 10:448-450.
			\bibitem{Chen_21} Chen M, Weng Z. Alternating projection method for doubly stochastic inverse eigenvalue problems with partial eigendata. Comp. Appl. Math. 2021; 40:164.
			\bibitem{Chu_1998} Chu MT, Guo G. A numerical method for the inverse stochastic spectrum problem, SIAM J. Matrix Anal. Appl. 1998; 19(4):1027-1039.
			\bibitem{Golub_05} Chu MT, Golub GH. Inverse Eigenvalue Problems: Theory, Algorithms, and Application, Oxford Science Publications, Oxford University Press; 2005.
			\bibitem{Douik_Graph} Douik A, Hassibi B. A Riemannian approach for graph-based clustering by doubly stochastic matrices, in Proc. IEEE Statistical Signal Process. Workshop 2018; 1:806-810.
			\bibitem{Douik_geo} Douik A, Hassibi B. Manifold optimization over the set of doubly stochastic matrices: A second-order geometry, IEEE Trans. Signal Process. 2019; 67(22):5761-5774.
			\bibitem{Drusvyatskiy_2015} Drusvyatskiy D,  Ioffe AD, Lewis AS. Transversality and alternating projections for nonconvex sets. Found. Comput. Math. 2015; 15(6):1637-1651.
			\bibitem{Egleston_2004}Egleston PD, Lenker TD,  Narayan SK. The nonnegative inverse eigenvalue problem, Linear Algebra Appl. 2004; 379:475-490.
			\bibitem{dykstra_algo}Dykstra RL. An algorithm for restricted least-squares regression, J. Amer. Statist. Assoc. 1983; 78:837-842.	
			\bibitem{Chu} Chu MT, Driessel KR. Constructing symmetric nonnegative matrices with prescribed eigenvalues by differential equations, SIAM J. Math. Anal. 1991; 22:1372-87.
			\bibitem{Converge_MAP} Fiorot JCh, Huard P. Composition and union of general algorithms of optimization, Mathematical Programming Study 1979; 10:69-85.
			\bibitem{Gaffke}Gaffke N, Mathar R. A cyclic projection algorithm via duality, Metrika 1989; 36:29-54.
			\bibitem{Golub} Golub GH, Van Loan CF. Matrix Computations Johns Hopkins Studies in the Mathematical Sciences; 2013.
			\bibitem{Robert_1980}Hardt RM. Semi-algebraic local-triviality in semi-algebraic mappings, Amer. J. Math. 1980; 102(2):291-302.	
			\bibitem{Hardy}Hardy GH, Littlewood JE, Polya G. Inequalities, Cambridge University Press, Cambridge; 1952.
			\bibitem{Halperin_62}Halperin I. The product of projection operators. Acta Sci. Math. (Szeged), 1962; 23:96-99.
			\bibitem{M_cpct_manifold}Helmke U, Moore JB. Optimization and Dynamical Systems, Springer-Verlag, London; 1994.
			\bibitem{Horn} Horn RA, Johnson CR. Matrix Analysis, Cambridge University Press; 1985.
			\bibitem{John_1981} Johnson CR. Row stochastic matrices similar to doubly stochastic matrices, Linear and Multilinear Algebra 1981; 10:113-130.
			\bibitem{Khoury} Khoury RN. Closest matrices in the space of generalized doubly stochastic matrices, journal of mathematical analysis and applications 1998; 222:562-568.
			\bibitem{Lewis_2008} Lewis AS, Malick J. Alternating projections on manifolds, Mathematics of Operations Research 2008; 33:216-234.
			\bibitem{Minc_88} Minc H. Non-negative matrices, Berlin Press, New York; 1988.
			\bibitem{M1} Mourad B. An inverse problem for symmetric doubly stochastic matrices, Inverse Problems 2003; 19:821-831.
			\bibitem{M2} Mourad B. A note on the boundary of the set where the decreasingly ordered spectra of symmetric doubly stochastic matrices lie,
               Linear Algebra Appl. 2006; 416:546-558
            \bibitem{M3} Mourad B. On a spectral property of doubly stochastic matrices and its application to their inverse eigenvalue problem, Linear Algebra Appl. 2012; 436:3400-3412.
			\bibitem{M4} Mourad B, Abbas H, Mourad A, Ghaddar A, Kaddoura I. An algorithm for constructing doubly stochastic matrices for the inverse eigenvalue problem, Linear Algebra Appl. 2013; 439:1382-1400.
 \bibitem{M5} Mourad B.    Generalization of some results concerning eigenvalues of a certain class of matrices and some applications
Linear and Multilinear Algebra 2013;  61;1234-1243.
			\bibitem{M6}Nader R, Mourad B, Bretto A, Abbas H. A note on the real inverse spectral problem for doubly stochastic matrices, Linear Algebra Appl. 2019; 569:206-240.
			\bibitem{Orsi} Orsi R. Numerical methods for solving inverse eigenvalue problems for nonnegative matrices, SIAM J. Matrix Anal. 2006; 28:190-212.
			\bibitem{Genetic} He MX, Petoukhov SV, Ricci PE. Genetic code, hamming distance and stochastic matrices. Bull. Math. Biol. 2006; 66:1405-1421.	
			\bibitem{Rammal_1} Rammal K, Mourad B, Abbas H,  Issa H. On the doubly stochastic realization of spectra. Banach J. Math. Anal. 16, 49 (2022). 
			\bibitem{Rivaz} Rivaz A, Moghadam M, Zadeh S. Doubly Stochastic Interval Matrices, Cankaya University Journal of science and Engineering 2015; 12:12-19.	
			\bibitem{Seroul_00} Séroul R. Programming for Mathematicians, Berlin Springer-Verlag; 2000.
			\bibitem{Smo} Smoktunowicz A, Kozera R, Oderda G. Efficient numerical algorithms for constructing orthogonal generalized doubly stochastic matrices, Applied Numerical Mathematics 2019; 142:16-27.
\bibitem{Convergence_of_MAP}Tropp JA, Dhillon IS, Heath RW, Strohmer T. Designing structured tight frames via an alternating projection method, IEEE Trans. Inf. Theory 2008; 51:188-209.
			
			\bibitem{Von_Neumann_33} Von Neumann J. Functional Operators. Vol. II. The Geometry of Orthogonal Spaces. Annals of Math. Studies 22, 1950. Princeton University Press, Princeton, NJ. This is a reprint of mimeographed lecture notes first distributed in 1933.
			\bibitem{CG_Method_2020} Wang Y, Zhao Z, Bai ZJ. Riemannian Newton-CG methods for constructing a positive doubly stochastic matrix from spectral data; 2020.
			\bibitem{Zarantonello_71} Zarantonello EH. Projections on convex sets in Hilbert space and spectral theory. In E. Zarantonello, editor, Contributions to Nonlinear Functional Analysis 1971; 1:237-424.
		\end{thebibliography}
	\end{document}